\documentclass[11pt]{article}

\usepackage[dvips,a4paper,hmargin=2.75cm,vmargin=3cm]{geometry}
\usepackage[utf8]{inputenc}
\usepackage{hyperref}
\usepackage{amssymb,amsmath,amsthm,mathtools}
\usepackage{enumitem}
\usepackage{cite}
\usepackage[mathscr]{euscript}
\usepackage[T1]{fontenc}
\usepackage{lmodern}
\usepackage{fontawesome}
\usepackage{microtype}
\usepackage{graphicx}
\usepackage[normalem]{ulem}
\usepackage{authblk}
\usepackage[small]{titlesec}
\usepackage{cite}

\newcommand{\mailto}[1]{\href{mailto:#1}{\nolinkurl{#1}}}

\hypersetup{colorlinks=true,linkcolor=blue,citecolor=blue,
			pdfinfo={
			Title   = {One-dimensional Stark operators
						in the half-line},
			Author  = {Julio H. Toloza, Alfredo Uribe},
			Subject = {Manuscript}}
			}

\titleformat*{\subsection}{\bfseries\itshape}

\providecommand{\keywords}[1]
{
\noindent{\small
	\textbf{Keywords:} #1}
}

\providecommand{\msc}[1]
{
\noindent{\small
	\textbf{2010 MSC:} #1}
}

\newtheorem{theorem}{Theorem}[section]

\newtheorem{lemma}[theorem]{Lemma}

\newtheorem{proposition}[theorem]{Proposition}
\newtheorem*{thm}{Theorem}

\theoremstyle{definition}

\newtheorem*{acknowledgments}{Acknowledgments}

\newcommand{\R}{{\mathbb R}}
\newcommand{\N}{{\mathbb N}}

\newcommand{\C}{{\mathbb C}}

\newcommand{\cD}{\mathcal{D}}
\newcommand{\cE}{\mathcal{E}}
\newcommand{\cF}{\mathcal{F}}

\DeclareMathOperator{\ai}{Ai}
\DeclareMathOperator{\bi}{Bi}

\DeclareMathOperator*{\re}{Re}
\DeclareMathOperator*{\im}{Im}
\DeclareMathOperator{\dom}{\cD}

\DeclareMathOperator{\spec}{\sigma}

\newcommand{\abs}[1]{\left\lvert #1 \right\rvert}
\newcommand{\abss}[1]{\bigl\lvert #1 \bigr\rvert}
\newcommand{\norm}[1]{\left\lVert #1 \right\rVert}

\title {\bf One-dimensional Stark operators in the half-line}

\author[1]{Julio H. Toloza%
			\thanks{\faEnvelopeO\, {julio.toloza@uns.edu.ar (corresponding author)}}%
			}
\author[2]{Alfredo Uribe%
			\thanks{\faEnvelopeO\, {alfredo.uribe.83@gmail.com}}%
			}
\affil[1]{Instituto de Matemática (INMABB)\\
		 Departamento de Matemática\\
		 Universidad Nacional del Sur (UNS) - CONICET\\
		 Bahía Blanca\\
		 Argentina}
\affil[2]{Departamento de Matemáticas\\
		 Universidad Autónoma Metropolitana\\
		 Av. San Rafael Atlixco 186\\
		 Col. Vicentina, Iztapalapa, C.P. 09340, México D.F.}

\begin{document}

\date{}
\maketitle

\begin{abstract}
We obtain asymptotic formulas for the spectral data of
perturbed Stark operators associated with the differential expression
\[
-\frac{d^2}{dx^2} + x + q(x),
\quad
x\in [0,\infty),
\quad
q\in L^1(0,\infty),
\]
and having either Dirichlet or Neumann boundary condition at the origin.
\end{abstract}

\bigskip

\msc{
34E10,	
34L15,	
81Q05,	
81Q10	
}


\section{Introduction and statement of results}

Self-adjoint operators of the form
\[
-\frac{d^2}{dx^2} + f(x) + q(x), \quad x\in(0,\infty),
\]
with domain in $L^2(\R_+)$, occur naturally in the context of quantum-mechanical
operators with spherical symmetry; here $q$ plays the role of a small perturbation of
$f$ in some suitable sense. The spectral analysis of this kind of operators have attracted
considerable attention for various choices of the dominant term $f$, usually
in connection with well-known special functions. Most remarkable among them are the
investigations concerning perturbed Bessel operators
\cite{ahm2,car,car2,guliyev,kst,kst2,kt,kt2}
(corresponding to $f(x)=l(l+1)x^{-2}$, $l\ge -1/2$), and perturbed harmonic oscillator
in the half-line \cite{chelkak2.5,chelkak3} (in this case $f(x)=x^{2}$); the latter is
closely related to the spectral analysis of perturbed harmonic oscillator in the whole
real line \cite{chelkak1,chelkak2}.

In this paper we consider self-adjoint operators associated with a differential
expression of the form
\begin{equation*}
\tau = -\frac{d^2}{dx^2} + x + q(x),
\quad
x\in [0,\infty),
\end{equation*}
acting in the space $L^2(\R_+)$, where $q$ is a real-valued function that lies in
$L^1(\R_+)$.
Self-adjoint operators are defined by adjoining to $\tau$ a standard boundary condition
(see Section~\ref{sec:preliminaries}) at the left endpoint. For the sake of brevity,
we only consider Dirichlet ($\varphi(0)=0$) and Neumann ($\varphi'(0)=0$) boundary
conditions; let $H^D$ and $H^N$ denote the corresponding self-adjoint operators.

Clearly the unperturbed case $q\equiv0$ can be solved explicitly. For in this case
a square-integrable solution to the associated eigenvalue problem is given by
the Airy function of the first kind $\ai(z)$ so
\[
\spec(H_0^D)
	= \{-a_k\}_{k\in\N}
\quad\text{and}\quad
\spec(H_0^N)
	= \{-a_k'\}_{k\in\N},
\]
where $a_k$ and $a_k'$ denote the zeros of $\ai(z)$ and its
derivative $\ai'(z)$, respectively. The corresponding set of norming constants
$\{\nu^D_{0,k}\}_{k\in\N}$ and $\{\nu^N_{0,k}\}_{k\in\N}$ are then given by
\[
\frac{1}{\nu^D_{0,k}} = \frac{\norm{\ai(\cdot+a_k)}^2}{\left(\ai'(a_k)\right)^2} = 1
\quad\text{and}\quad
\frac{1}{\nu^N_{0,k}} = \frac{\norm{\ai(\cdot+a'_k)}^2}{\left(\ai(a'_k)\right)^2} = -a'_k.
\]
The related results for arbitrary $q$ are stated in Theorems~\ref{thm-dirichlet-eigenvalues},
\ref{thm:norming-constant-dirichlet}, \ref{thm-mixed-bc-eigenvalues} and
\ref{thm:norming-constant-neumann}. They can be summarized as follows:

\begin{thm}
Suppose $q\in L^1(\R_+)$. Then the eigenvalues and norming constants
of $H^D$, the operator associated with $\tau$ and boundary condition $\varphi(0)=0$,
satisfy
\[
\lambda^D_k = \left(\tfrac32\pi\bigl(k-\tfrac14\bigr)\right)^{2/3}
			\left(1 + O(k^{-1})\right)
\quad\text{and}\quad
\frac{1}{\nu^D_k} = 1+o(1)
\]
as $k\to\infty$. Similarly, the eigenvalues and norming constants
of $H^N$ corresponding to the boundary condition $\varphi'(0)=0$ satisfy
\[
\lambda^N_k = \left(\tfrac32\pi\bigl(k-\tfrac34\bigr)\right)^{2/3}
			\left(1 + O(k^{-1})\right)
\quad\text{and}\quad
\frac{1}{\nu^N_k} = \left(\tfrac32\pi\bigl(k-\tfrac34\bigr)\right)^{2/3}
			\left(1 + o(1)\right)
\]
as $k\to\infty$.
\end{thm}

The direct spectral problem for
the one-dimensional Stark operator in the semi-axis, with Dirichlet boundary condition,
has also been treated recently in \cite{mk}, where the authors use transformation operator
methods and their results are valid under the more restrictive assumption
$q\in C^{(1)}[0,\infty)\cap L^1(\R_+,x^4 dx)$, $q(x)=o(x)$ as $x\to\infty$.
The corresponding inverse spectral problem is discussed in \cite{lk}.

Finally, it is worth mentioning that one-dimensional Stark
operators have been studied mostly when defined on the whole real line, see for
instance \cite{calogero,its,khanmamedov,katchalov,liu,sukhanov}.
As it is well-known, Stark operators on the real line are
characterized by the presence of  resonances; see \cite{korot1,korot2}
for some recent developments on this subject.


\section{Preliminaries}
\label{sec:preliminaries}

In what follows, we consider the differential expression
\[
\tau = -\frac{d^2}{dx^2} + x + q(x),
\quad
x\in [0,\infty),
\]
where $q\in L^1(\R_+)$ and it is real-valued.

By standard theory (see e.g. \cite[Ch.~6]{weidmann}), $\tau$ is in the
limit-circle case at $0$ and in the limit-point case at $\infty$. Hence (the
closure of) the minimal operator $H'$ defined by $\tau$ is symmetric and has
deficiency indices $(1,1)$.
Also, there exists a solution $\psi(z,x)$ to the eigenvalue equation
$\tau\varphi=z\varphi$,
real entire as a function of $z\in\C$ for every $x\in[0,\infty)$, such that
$\psi(z,\cdot)\in L^2(\R_+)$ for every $z\in\C$. This function is unique up to
multiplication
by a zero-free, real entire function of the spectral parameter $z$.

The self-adjoint extensions $H^\beta$ ($0\le\beta<\pi$) of $H'$ are
determined by imposing the usual boundary condition at $x=0$. Namely,
\begin{equation*}
\dom(H^\beta)
	= \left\{\begin{gathered}
			\varphi\in L^2(\R_+) : \varphi,\varphi'\in\text{AC}_\text{loc}([0,\infty)),\
			\tau\varphi\in L^2(\R_+),
			\\[1mm]
			\cos(\beta)\varphi(0) - \sin(\beta)\varphi'(0) = 0
			\end{gathered}\right\},
\quad
H^\beta\varphi = \tau\varphi.
\end{equation*}
Since $x+q(x)\to\infty$ as $x\to\infty$, it follows that $\spec(H^\beta)$ has only eigenvalues
of multiplicity one, possibly with a finite number of them being negative. Moreover,
\[
\spec(H^\beta)
	= \{\lambda\in\R: \cos(\beta)\psi(\lambda,0) - \sin(\beta)\psi'(\lambda,0)  = 0\}
\quad (0\le\beta<\pi).
\]
We henceforth suppose $\spec(H^\beta)$ is arranged as an increasing sequence, viz.,
$\spec(H^\beta)=\{\lambda_k^\beta\}_{k\in\N}$ with $\lambda^\beta_k<\lambda^\beta_{k+1}$.

In what follows we use the notation $'=\partial_x$ and $\dot{}=\partial_z$.
Along with the spectrum $\{\lambda_k^\beta\}_{k\in\N}$ one has the corresponding set of
norming constants $\{\nu_k^\beta\}_{k\in\N}$. In terms of $\psi(z,x)$, the norming
constants for Dirichlet ($\beta=0$) and Neumann ($\beta=\pi/2$) boundary conditions
are given by the formulas
\begin{equation*}
\frac{1}{\nu^D_k}
	= \frac{\norm{\psi(\lambda^D_k,\cdot)}^2}
		{\left[\psi'(\lambda^D_k,0)\right]^2}
	= - \frac{\dot{\psi}(\lambda^D_k,0)}{\psi'(\lambda^D_k,0)}
\quad\text{and}\quad
\frac{1}{\nu^N_k}
	= \frac{\norm{\psi(\lambda^N_k,\cdot)}^2}
		{\left[\psi(\lambda^N_k,0)\right]^2}
	= \frac{\dot{\psi}'(\lambda^N_k,0)}{\psi(\lambda^N_k,0)},
\end{equation*}
respectively. The second part of these equations follows from the identity
$W'(\eta,\dot{\eta})= -\eta^2$, which is valid for any solution to $\tau\eta=z\eta$.
We recall that the spectral data
$(\{\mu_k^\beta\}_{k\in\N},\{\nu_k^\beta\}_{k\in\N})$ are the poles
and residues of the Weyl function associated with $H^\beta$, and they determine
the potential $q$ by virtue of the Borg--Marchenko uniqueness theorem
\cite{gesztesy}.

As mentioned in the Introduction, the unperturbed case $q=0$ can be treated explicitly.
A solution to the equation $-\varphi'' + (x-z)\varphi$, belonging to $L^2(\R_+)$, is
\[
\psi_0(z,x) = \sqrt{\pi}\,\ai(x-z),
\]
where the factor $\sqrt{\pi}$ is included for convenience. It follows that
\[
\spec(H_{0}^{D}) = \{-a_k\}_{k\in\N}\quad\text{and}\quad
\spec(H_{0}^{N}) = \{-a_k'\}_{k\in\N}
\]
respectively, where the zeros of $\ai(z)$ and $\ai'(z)$ obey the asymptotic formulas
\begin{equation}
\label{eq:zeros-airy}
-a_k
		= \left(\tfrac32\pi\bigl(k-\tfrac14\bigr)\right)^{2/3}
			\left(1 + O(k^{-2})\right)
\end{equation}
and
\begin{equation*}
-a_k'
		= \left(\tfrac32\pi\bigl(k-\tfrac34\bigr)\right)^{2/3}
			\left(1 + O(k^{-2})\right)
\end{equation*}
as $k\to\infty$ (see \cite[\S 9.9(iv)]{nist}).

\begin{lemma}
\label{lem:g_Ag_B}
There exists a constant $C_0>0$ such that
\begin{equation*}
\abss{\ai(z)}
\le C_0\frac{g_A(z)}{1+\abs{z}^{1/4}} \quad\text{and}\quad
\abss{\ai'(z)}
\le C_0 \bigl(1 + \abs{z}^{1/4}\bigr) g_A(z),
\end{equation*}
for all $z\in\C$, where $g_A(z) = \exp(-\tfrac23\re z^{3/2})$.
\end{lemma}

\begin{proof}
Define $\zeta =\frac23 z^{3/2}$ with branch cut along $\R_-$.
According to \cite[\S 9.7(ii)]{nist}, the function $\text{\rm Ai}(z)$
satisfies the asymptotic expansions
\begin{equation}
\label{eq:ai-asymp-positive}
\ai(z)
	= \frac{e^{-\zeta}}{2\sqrt{\pi}z^{1/4}}
		\left[1+O(\zeta^{-1})\right],
	    \quad \abs{\arg(z)}\le\pi-\delta,
\end{equation}
and
\begin{equation}
\label{eq:ai-asymp-negative}
\ai(-z)
	= \frac{1}{\sqrt{\pi}z^{1/4}}
	  \left[\sin\left(\zeta+\tfrac{\pi}{4}\right)
	  	+O\left(\zeta^{-1}e^{\abs{\im\zeta}}\right)\right],
	  \quad \abs{\arg(z)}\le\tfrac{2\pi}{3} - \delta,
\end{equation}
as $\abs{z}\to\infty$.
These expansions are uniform for any given small $\delta>0$ and $\abs{z}\ge 1$.
In what follows we set $\delta=\pi/3$.
Since $\text{\rm Ai}(z)$ is an entire function, it follows that there exists $C_0>0$
such that
\[
\abs{\ai(z)}
	\le \frac{C_0}{1+\abs{z}^{1/4}}\times
	\begin{cases}
		\exp(-\tfrac23 \re z^{3/2}),&
		\arg(z)\in [-\frac{2\pi}{3},\frac{2\pi}{3}],
		\\[1mm]
		\exp(\tfrac23\abss{\im(-z)^{3/2}}),&
		\arg(z)\in(-\pi,-\frac{2\pi}{3})\cup(\frac{2\pi}{3},\pi].
		\end{cases}
\]
Thus, the bound on $\ai(z)$ follows after noticing that
$\abss{\im(-z)^{3/2}}=\abss{\re z^{3/2}}$ and $\abss{\re z^{3/2}}=-\re z^{3/2}$ if
$\arg(z)\in(-\pi,-\frac{2\pi}{3})\cup(\frac{2\pi}{3},\pi]$.
The bound on $\ai'(z)$ follows an analogous argument so the details are omitted.
\end{proof}

Lemma~\ref{lem:g_Ag_B} clearly implies
\begin{equation}
\label{eq:bound-psi-0}
\abs{\psi_0(z,x)}
\le C_0\frac{e^{-\frac23\re(x-z)^{3/2}}}{1+\abs{x-z}^{1/4}}
\quad\text{and}\quad
\abs{\psi_0'(z,x)}
\le C_0 \bigl(1+\abs{x-z}^{1/4}\bigr) e^{-\frac23\re(x-z)^{3/2}}
\end{equation}
with $(x,z)\in\R_+\times\C$.
Later we will make use of a linearly independent solution to $-\varphi'' + (x-z)\varphi$.
An obvious choice is given by the Airy function of the second kind
\[
\theta_0(z,x) = \sqrt{\pi}\,\bi(x-z).
\]
However, it will be more convenient to use an independent solution of the form
\begin{equation*}
\theta_\pm(z,x)
	 = \theta_0(z,x) \mp i \psi_0(z,x)
	 = 2\sqrt{\pi} e^{\mp i\pi/6}\ai((x-z)e^{\mp i2\pi/3})
\end{equation*}
(in the context of the present work any of these two functions is equally good).
According to \cite[\S 9.2(iv)]{nist}, one has $W(\psi_0(z),\theta_\pm(z))\equiv 1$.
Moreover, since $\re(ze^{\pm i2\pi/3})^{3/2} = -\re z^{3/2}$, we have the bounds
\begin{equation}
\label{eq:bound-theta-pm}
\abs{\theta_\pm(z,x)}
	\le 2C_0\frac{e^{\frac23\re(x-z)^{3/2}}}{1+\abs{x-z}^{1/4}}
\quad\text{and}\quad
\abs{\theta'_\pm(z,x)}
	\le 2C_0 \bigl(1+\abs{x-z}^{1/4}\bigr) e^{\frac{2}{3}\re(x-z)^{3/2}}.
\end{equation}

\begin{lemma}
\label{lem:stupid-lemma}
The map $x\mapsto g_A(x-z)$, $x\in\R_+$, is decreasing whenever $z\in\C\setminus\R$.
If $z\in\R$, then $g_A(x-z)$ is constant (equal to 1) for $x\in[0,z]$ and decreasing
for $x\in(z,\infty)$.
\end{lemma}
\begin{proof}
Suppose $z\in\C_-$. A simple computation shows that, given $x\in\R$, there
exists a unique $\gamma\in(0,\pi)$ such that
\begin{equation*}
x-z= \frac{\abs{\im z}}{\sin\gamma}e^{i\gamma}.
\end{equation*}
Then,
\begin{equation*}
\re(x-z)^{3/2} = \abs{\im z}^{3/2}\frac{\cos\frac32\gamma}{(\sin\gamma)^{3/2}}.
\end{equation*}
The right hand side of the last equation is decreasing
as a function of $\gamma$. But the map $x\mapsto \gamma$ is also decreasing so
the map $x\mapsto \re(x-z)^{3/2}$ is increasing. This in turn implies the assertion.
Clearly, a similar reasoning works if $z\in\C_+$. The statement is obvious for $z\in\R$.
\end{proof}

\section{Main results}
\label{sec:main}

\subsection{Adding a perturbation}

We look for a solution to the eigenvalue equation
$\tau\varphi = z\varphi$, with $q\in L^1(\R_+)$, that is real entire with
respect to the spectral parameter $z\in\C$ and lies in $L^2(\R_+)$.
To this end we introduce the auxiliary function
\begin{equation*}
\omega(z) = \int_0^\infty\frac{\abs{q(x)}}{\sqrt{1+\abs{x-z}}}\,dx.
\end{equation*}
Clearly, $\omega(z)$ is well defined for all $z\in\C$. Moreover,
$\omega(z)$ is well defined under the weaker assumption
$q\in L^1(\R_+,(1+x)^{-1/2}dx)$. However, our hypothesis on $q$ give us control
on the decay of $\omega(z)$ as it is shown next.

\begin{lemma}
\label{rem:maybe-true}
Assume $q\in L^1(\R_+)$. Then $\omega(z)\to 0$ as $z\to\infty$.
\end{lemma}
\begin{proof}
Given $\varepsilon>0$, choose $x_* > 0$ and $\mu_* > x_*$ such that
\[
\int_{x_*}^\infty \abs{q(x)} dx < \frac{\varepsilon}{2}
\quad\text{and}\quad
\frac{1}{\sqrt{\mu_* - x_*}} < \frac{\varepsilon}{2\norm{q}_1}.
\]
Suppose $\abs{\im(z)}>\mu_*$. Then
$\abs{x-z}>\mu_*$ for any $x>0$. Hence,
\begin{equation*}
\frac{1}{\sqrt{1+\abs{x-z}}}
	\le \frac{1}{\sqrt{\mu_*}}
	< \frac{\epsilon}{2\norm{q}_{1}}
\end{equation*}
for all $x\in\R_*$, which in turn implies $\omega(q,z)<\epsilon$.
A similar reasoning applies when $\abs{\im(z)} \le \mu_*$ and $\re(z) < - \mu_*$.
Finally, suppose that $\abs{\im(z)} \le \mu_*$ and $\re(z) > \mu_*$.
Since $\omega(q,z)\le \omega(q,\re(z))$, it suffices to consider $z=\mu\in\R$
with $\mu > \mu_*$. Then,
\begin{equation*}
\omega(q,\mu) 
	< \frac{1}{\sqrt{1+\abs{x_*-\mu_*}}}\int_0^{x_*} \abs{q(x)} dx 
		+ \int_{x_*}^\infty \abs{q(x)} dx
	< \varepsilon.
\end{equation*}
Thus, we have shown that
\(
\omega(q,z) < \varepsilon 
\)
whenever $\abs{\re(z)} + \abs{\im(z)} > \mu_*$.
\end{proof}

In what follows $C$ denotes a generic positive constant.

\begin{proposition}
\label{prop:bad-estimate}
Suppose $q\in L^1(\R_+,(1+x)^{-1/2}dx)$. Then, the eigenvalue
equation $\tau\varphi = z\varphi$ admits a solution $\psi(z,x)$, real entire with respect
to $z$, such that:
\begin{enumerate}[label={(\roman*)}]
\item $\psi(z,x)$ solves the Volterra integral equation
\begin{equation}
\label{eq:volterra-1}
\psi(z,x) = \psi_0(z,x) - \int_x^{\infty}J_0(z,x,y)q(y)\psi(z,y)dy,
\end{equation}
where
\[
J_0(z,x,y) = \psi_0(z,y)\theta_0(z,x)-\psi_0(z,x)\theta_0(z,y),
\]
and satisfies the estimates
\begin{align}
\abs{\psi(z,x)}
	&\le C e^{C\omega(z)}
	\frac{g_A(x-z)}{1+\abs{x-z}^{1/4}}\quad\text{and}\nonumber
\\[1mm]
\abs{\psi(z,x)-\psi_0(z,x)}
	&\le C\omega(z)e^{C\omega(z)}
	\frac{g_A(x-z)}{1+\abs{x-z}^{1/4}}.
	\label{eq:estimate-diff-phi}
\end{align}

\item Moreover, $\psi'(z,x)$ obeys the equation
\begin{equation*}
\psi'(z,x) = \psi_0'(z,x) - \int_x^{\infty}\partial_xJ_0(z,x,y)q(y)\psi(z,y)dy
\end{equation*}
and satisfies the estimates
\begin{equation}
\abs{\psi'(z,x)-\psi'_0(z,x)}
	\le C\omega(z)e^{C\omega(z)}\bigl(1+\abs{x-z}^{1/4}\bigr)g_A(x-z).
	\label{eq:estimate-diff-phi-neumann}
\end{equation}
\end{enumerate}
\end{proposition}

\begin{proof}
For $n\in\N$ define
\begin{equation*}
\psi_n(z,x) = - \int_x^\infty J_0(z,x,y)q(y)\psi_{n-1}(z,y)dy.
\end{equation*}
Then,
\begin{equation}
\label{eq:recursive}
\abs{\psi_n(z,x)}
\le \int_x^\infty \abs{J_0(z,x,y)}\abs{q(y)}\abs{\psi_{n-1}(z,y)}dy.
\end{equation}
Next, we note that
\[J_0(z,x,y)
	= \pm i\left[\psi_0(z,x)\theta_\pm(z,y) - \psi_0(z,y)\theta_\pm(z,x)\right]
\]
(the choice of sign is irrelevant). Then, recalling
\eqref{eq:bound-psi-0} and \eqref{eq:bound-theta-pm}, \eqref{eq:recursive} yields
\begin{multline*}
\abs{\psi_n(z,x)}
\le 2C_0^2\frac{g_A(x-z)}{1+\abs{x-z}^{1/4}}
		\int_x^{\infty}\frac{\abs{q(y)}}{1+\abs{y-z}^{1/4}}g_{\#}(y-z)
		\abs{\psi_{n-1}(z,y)}dy
\\[1mm]
+   2C_0^2\frac{g_{\#}(x-z)}{1+\abs{x-z}^{1/4}}
		\int_x^{\infty}\frac{\abs{q(y)}}{1+\abs{y-z}^{1/4}}g_A(y-z)
		\abs{\psi_{n-1}(z,y)}dy.
\end{multline*}
where $g_{\#}(z):=1/g_A(z)$.
We claim that every $\psi_n(z,x)$ is real entire with respect to the spectral parameter
and satisfies the estimate
\begin{equation}
\label{eq:estimate-diff-phi_n}
\abs{\psi_n(z,x)}
\le \frac{4^n}{n!}C_0^{2n+1} \omega^n(z)\frac{g_A(x-z)}{1+\abs{x-z}^{1/4}}.
\end{equation}
From this it will follow that
\[
\psi(z,x) = \sum_{n=0}^\infty \psi_n(z,x)
\]
converges uniformly on compact subsets of $\C$ to the solution with the desired
properties.

First, consider $n=1$. Then, we have
\begin{multline*}
\abs{\psi_1(z,x)}
\le 2C_0^3\frac{g_A(x-z)}{1+\abs{x-z}^{1/4}}
		\int_x^{\infty}\frac{\abs{q(y)}}{(1+\abs{y-z}^{1/4})^2}g_{\#}(y-z)g_A(y-z)dy
\\[1mm]
+   2C_0^3\frac{g_{\#}(x-z)}{1+\abs{x-z}^{1/4}}
		\int_x^{\infty}\frac{\abs{q(y)}}{(1+\abs{y-z}^{1/4})^2}(g_A(y-z))^2dy.
\end{multline*}
Clearly, $(g_Ag_{\#})(x-z)\equiv 1$. Also, Lemma~\ref{lem:stupid-lemma}
implies $g_A(y-z)\le g_A(x-z)$ for all $y\in[x,\infty)$. Hence,
\begin{multline*}
\abs{\psi_1(z,x)}
\le 2C_0^3\frac{g_A(x-z)}{1+\abs{x-z}^{1/4}}
		\int_x^{\infty}\frac{\abs{q(y)}}{(1+\abs{y-z}^{1/4})^2}dy
\\[1mm]
+   2C_0^3\frac{(g_{\#} g_Ag_A)(x-z)}{1+\abs{x-z}^{1/4}}
		\int_x^{\infty}\frac{\abs{q(y)}}{(1+\abs{y-z}^{1/4})^2}dy,
\end{multline*}
that is,
\[
\abs{\psi_1(z,x)}
\le 4C_0^3\frac{g_A(x-z)}{1+\abs{x-z}^{1/4}}
		\int_x^{\infty}\frac{\abs{q(y)}}{(1+\abs{y-z}^{1/4})^2}dy.
\]

For arbitrary $n\in\N$ we use the identity
\begin{equation}
\label{eq:identity-Poschel}
\int_x^\infty\int_{y_1}^\infty\cdots\int_{y_{n-1}}^\infty\prod_{l=1}^n
h(y_l)\,dy_1\cdots dy_n
	= \frac{1}{n!}\left[\int_x^\infty h(y) dy\right]^n
\end{equation}
to obtain
\begin{equation}
\label{eq:inductive-useful-bound}
\abs{\psi_n(z,x)}
\le \frac{4^n}{n!}C_0^{2n+1}\frac{g_A(x-z)}{1+\abs{x-z}^{1/4}}
		\left[\int_x^{\infty}\frac{\abs{q(y)}}{(1+\abs{y-z}^{1/4})^2}dy\right]^n
\end{equation}
which in turn implies \eqref{eq:estimate-diff-phi_n}. Then (i) follows after a
suitable choice for the constant $C$.

The proof of (ii) is omitted since it arises from an analogous reasoning.
\end{proof}

Clearly, \eqref{eq:estimate-diff-phi} implies that $\psi(z,x)$ so constructed
belongs to the domain of the maximal operator $H$.

The asymptotic analysis of the norming constants depends also on the following estimates.

\begin{proposition}
Suppose $q\in L^1(\R_+)$. Then, $\psi(z,x)$ satisfies
\begin{equation}
\label{prop:another-estimate}
\abs{\dot{\psi}(z,x) - \dot{\psi}_0(z,x)}
	\le C e^{C\norm{q}}
	\left(\bigl(1 + \abs{x-z}^{1/4}\bigr)\omega(z) + \frac{\norm{q}}
	{1 + \abs{x-z}^{1/4}}\right)g_A(x-z).
\end{equation}
Also,
\begin{equation}
\label{prop:last-estimate}
\abs{\dot{\psi}'(z,x) - \dot{\psi}'_0(z,x)}
	\le C e^{C\norm{q}}
	\left(\bigl(1 + \abs{x-z}^{1/4}\bigr)\norm{q}^2
		+ \frac{\abs{x-z}}{1 + \abs{x-z}^{1/4}}\omega(z)\right)g_A(x-z).
\end{equation}
\end{proposition}

\begin{proof}
From \eqref{eq:volterra-1} we see that $\dot{\psi}(z,x)$ is a solution to the integral
equation
\[
\dot{\psi}(z,x) = \dot{\psi}_0(z,x)
	- \int_{x}^{\infty} \partial_z J_0(z,x,y) q(y) \psi(z,y) dy
	- \int_{x}^{\infty} J_0(z,x,y) q(y) \dot{\psi}(z,y) dy.
\]
Let $\{\eta_k(z,x)\}_{k\in\N}$, be solutions to the recursive equation
\[
\eta_k(z,x) =
	- \int_{x}^{\infty} \partial_z J_0(z,x,y) q(y) \psi_{k-1}(z,y) dy
	- \int_{x}^{\infty} J_0(z,x,y) q(y) \eta_{k-1}(z,y) dy,
\]
where $\{\psi_k(z,x)\}_{k\in\N}$ are defined in the proof of
Proposition~\ref{prop:bad-estimate} and $\eta_0(z,x) := \dot{\psi}_0(z,x)$. Using
induction one can show that
\begin{multline*}
\abs{\eta_k(z,x)}
	\le \frac{4^k C_0^{2k+1}}{k!}
	\left(\bigl(1 + \abs{x-z}^{1/4}\bigr)
	\left(\int_x^\infty\frac{\abs{q(y)}}{(1+\abs{y-z}^{1/4})^2}dy\right)^k\right.
	\\
	+\left. \frac{2k}{1+\abs{x-z}^{1/4}}\left(\int_x^\infty\abs{q(y)}dy\right)^k\right)
	g_A(x-z),
\end{multline*}
hence
\[
\abs{\eta_k(z,x)}
	\le \frac{4^k C_0^{2k+1}}{k!}
	\left(\bigl(1 + \abs{x-z}^{1/4}\bigr)
	\omega(z)^k
	+ \frac{2k}{1+\abs{x-z}^{1/4}}\norm{q}^k\right)
	g_A(x-z).
\]
It follows that
\[
\dot{\psi}(z,x) = \sum_{k=0}^\infty \eta_k(z,x)
\]
(the convergence being uniform on compact subsets of $\C$) which in turn implies the assertion.
The proof of the second inequality follows from an analogous reasoning.
\end{proof}

\subsection{Dirichlet boundary condition}

Define the contours
\begin{equation*}
\cE^m := \left\{z\in\C: \abs{\zeta} = \bigl(m+\tfrac14\bigr)\pi\right\},
\quad
\cE_k := \left\{z\in\C: \abs{\zeta - \bigl(k-\tfrac14\bigr)\pi}=\tfrac{\pi}{2}\right\},
\quad
m,k\in\N.
\end{equation*}
In view of \eqref{eq:zeros-airy}, every $\cE_k$ encloses one and only one zero
of $\ai(-\lambda)$, at least for $k$ sufficiently large.

\begin{lemma}
\label{lem:bound-for-gA}
There exists $m_0, k_0\in\N$ such that, for every $m\ge m_0$ and $k\ge k_0$, the
following statement holds true:
\begin{equation}
\label{eq:bound-for-gA}
\frac{g_A(-z)}{1+\abs{z}^{1/4}}
	< 8\sqrt{\pi} \abs{\text{\rm Ai}(-z)},
\end{equation}
whenever $z\in\cE^m$ or $z\in\cE_k$.
\end{lemma}
\begin{proof}
Let us begin by recalling \eqref{eq:ai-asymp-positive} and \eqref{eq:ai-asymp-negative}
in more precise terms:
\begin{align}
\label{eq:ai-also-detailed}
\ai(z)
	= \frac{e^{-\zeta}}{2\sqrt{\pi}z^{1/4}}
		\left[1+W_1(z)\right],
		\quad \abs{\arg(z)}\le\tfrac{2\pi}{3},
		\quad \abs{z}\ge 1,
\\[1mm]
\label{eq:ai-detailed}
\ai(-z)
	= \frac{1}{\sqrt{\pi}z^{1/4}}
		\left[\sin\left(\zeta+\tfrac{\pi}{4}\right)+W_2(z)\right],
		\quad \abs{\arg(z)}\le\tfrac{\pi}{3},
		\quad \abs{z}\ge 1,
\end{align}
where the functions $W_1(z)$ and $W_2(z)$ satisfy
\begin{align}
\abs{\frac{W_1(z)}{\zeta^{-1}}}\le D_1,
		\quad \abs{\arg(z)}\le\tfrac{2\pi}{3},
		\quad \abs{z}\ge 1,\nonumber
\\[1mm]
\abs{\frac{W_2(z)}{\zeta^{-1}e^{\abs{\im\zeta}}}}\le D_2,
		\quad \abs{\arg(z)}\le\tfrac{\pi}{3},
		\quad \abs{z}\ge 1.\label{eq:w2}
\end{align}

There exists $k_0\in\N$ such that, for all $k\ge k_0$, $z\in\cE_k$ implies
$\re z\ge 1$ and $\arg(z)\in(-\frac{\pi}{3},\frac{\pi}{3})$ so
$\arg(-z)\in(-\pi,-\frac{2\pi}{3})\cup(\frac{2\pi}{3},\pi]$. Since in this case
$\abss{\im z^{3/2}}=-\re (-z)^{3/2}$, one has
\begin{equation*}
\frac{g_A(-z)}{1+\abs{z}^{1/4}}
	=   \frac{e^{\abs{\im\zeta}}}{1+\abs{z}^{1/4}}
	\le \frac{e^{\abs{\im(\zeta+\frac{\pi}{4})}}}{\abs{z}^{1/4}}
\end{equation*}
for all $z\in\cE_k$ and $k\ge k_0$. By a well-known result
(see \cite[Ch.~2, Lemma 1]{poeschel}),
\begin{equation*}
\abs{w-n\pi}\ge \frac{\pi}{4} \implies e^{\abs{\im w}}<4\abs{\sin w}
\end{equation*}
for all integer $n$. Hence,
\begin{equation}
\label{eq:bound-1}
\frac{g_A(-z)}{1+\abs{z}^{1/4}}
	< 4 \frac{\abs{\sin(\zeta+\frac{\pi}{4})}}{\abs{z}^{1/4}}
\end{equation}
for all $z\in\cE_k$ and $k\ge k_0$.
On the other hand, since $\abs{\sin(\zeta+\tfrac{\pi}{4})}\ge d>0$ for all $z\in\cE_k$
and $k\ge k_0$, \eqref{eq:ai-detailed} implies
\begin{equation*}
\abs{\ai(-z)}
	\ge \frac{\abs{\sin(\zeta+\frac{\pi}{4})}}{\sqrt{\pi}\abs{z}^{1/4}}
	\abs{1-\frac{\abs{W_2(z)}}{\abs{\sin(\zeta+\tfrac{\pi}{4})}}}.
\end{equation*}
However,
\begin{equation*}
\frac{\abs{W_2(z)}}{\abs{\sin(\zeta+\tfrac{\pi}{4})}}
	\le \frac{e^{\abs{\im\zeta}}}{\abs{\zeta}}\frac{D_2}{d},
\end{equation*}
and note that $\abs{\im\zeta}\le \pi/2$ if $z\in\cE_k$.
Thus, by increasing $k_0$ if necessary, we have
\begin{equation*}
\abs{\ai(-z)}
	\ge \frac{\abs{\sin(\zeta+\frac{\pi}{4})}}{2\sqrt{\pi}\abs{z}^{1/4}},
\end{equation*}
for all $z\in\cE_k$ with $k\ge k_0$.

The proof concerning $\cE^m$ is analogous: Suppose $m_0=k_0$. Then, by the previous
argument, \eqref{eq:bound-for-gA} holds for $z\in\cE^m$ within the sector
$\arg(z)\in[-\frac{\pi}{3},\frac{\pi}{3}]$, for $m\ge m_0$. Within the sector
$\arg(-z)\in[-\frac{2\pi}{3},\frac{2\pi}{3}]$, we have ($\eta:=\frac23(-z)^{3/2}$)
\begin{equation*}
\frac{g_A(-z)}{1+\abs{z}^{1/4}}
	\le \frac{e^{-\re\eta}}{\abs{z}^{1/4}}
\end{equation*}
and, due to \eqref{eq:ai-also-detailed},
\begin{equation*}
\abs{\ai(-z)}
	\ge \frac{e^{-\re\eta}}{2\sqrt{\pi}\abs{z}^{1/4}}
		\abs{1-\abs{W_1(-z)}}.
\end{equation*}
Finally, using \eqref{eq:w2} ---and increasing $m_0$ if required---,
we have $1-\abs{W_1(-z)}\ge 1/4$ whenever $\abs{z}\ge m_0$.
\end{proof}

\begin{theorem}
\label{thm-dirichlet-eigenvalues}
Suppose $q\in L^1(\R_+)$. Then,
the eigenvalues of $H^D$ satisfy
\[
\lambda^D_k = \left(\tfrac32\pi\bigl(k-\tfrac14\bigr)\right)^{2/3}
			\left(1 + O(k^{-1})\right),\quad k\to\infty.
\]
\end{theorem}
\begin{proof}
Abbreviate
\begin{equation*}
\psi_0(z) := \psi_0(z,0),\quad  \psi(z) := \psi(z,0), \quad \mu_k:=-a_k.
\end{equation*}
Since $\sup_{z\in\C}\omega(z)<\infty$, Proposition~\ref{prop:bad-estimate} yields
\begin{equation*}
\abs{\psi(z)-\psi_0(z)}
	\le C \omega(z)\frac{g_A(-z)}{1+\abs{z}^{1/4}},
\end{equation*}
after redefining the constant $C$. Due to Lemma~\ref{rem:maybe-true}, there exists
$k_1\in\N$ such that $\omega(z)\le (8C)^{-1}$ whenever
$\abs{z}\ge(\frac32 \pi (k_1+\frac14))^{2/3}$. Then, by
Lemma~\ref{lem:bound-for-gA}, there exists $k_2\ge k_1$ such that
\begin{equation}
\label{eq:rouche}
\abs{\psi(z) - \psi_0(z)} < \abs{\psi_0(z)}
\end{equation}
for all $z\in\cE^{k_2}$; $k_2$ can be assumed large enough so $\cE^{k_2}$ encloses
all the (finitely many) negative zeros of $\psi(z)$.
Increase $k_2$ (if necessary) to ensure that \eqref{eq:rouche}
holds true for $z$ on every contour $\cE_{n}$ whenever $n\ge n_2$.
Then, in view of Rouché's theorem, we obtain
\[
\abs{\tfrac23(\lambda_k^D)^{3/2} - \tfrac23(-a_k)^{3/2}} \le \pi
\]
for sufficiently large $k$, whence the asymptotics for the eigenvalues follows.
\end{proof}

\begin{theorem}
\label{thm:norming-constant-dirichlet}
Suppose $q\in L^1(\R_+)$. Then the Dirichlet norming constants $\nu_k^D$ satisfies
\[
\frac{1}{\nu_k^D} = 1 + o(1)
\]
as $k\to\infty$.
\end{theorem}

\begin{proof}
Abbreviate
\[
\Delta_1(\lambda)
	:= \frac{\psi'(\lambda,0) - \psi'_0(\lambda,0)}{\sqrt{\pi}\ai'(-\lambda)},
\quad
\Delta_2(\lambda)
	:= \frac{\dot{\psi}(\lambda,0) - \dot{\psi}_0(\lambda,0)}
	{\sqrt{\pi}\ai'(-\lambda)}.
\]
It is straightforward to see that
\[
- \frac{\dot{\psi}(\lambda^D_k,0)}{\psi'(\lambda^D_k,0)}
	= 1 - \frac{\Delta_1(\lambda^D_k)}{1 + \Delta_1(\lambda^D_k)}
		- \frac{\Delta_2(\lambda^D_k)}{1 + \Delta_1(\lambda^D_k)}
\]
so it suffices to show that
\[
\Delta_1(\lambda^D_k)\to 0\quad \text{and}\quad \Delta_2(\lambda^D_k)\to 0
\]
as $k\to\infty$.

From Theorem~\ref{thm-dirichlet-eigenvalues} we obtain $\lambda_k^D = -a_k + O(k^{-1/3})$
thus
\[
\sqrt{\pi}\ai'(-\lambda^D_k)
	= (-1)^{k-1}\left(\tfrac32\pi\bigl(k-\tfrac14\bigr)\right)^{1/6}
	\left(1 + o(1)\right)
	= (-1)^{k-1} (\lambda^D_k)^{1/4}\left(1 + o(1)\right)
\]
as $k\to\infty$. On the other hand, from \eqref{eq:estimate-diff-phi-neumann}, we
have
\[
\abs{\psi'(\lambda^D_k,0) - \psi'_0(\lambda^D_k,0)}
	\le C \omega(\lambda^D_k) e^{C\omega(\lambda^D_k)} \abss{\lambda^D_k}^{1/4}
\]
hence the assertion on $\Delta_1(\lambda^D_k)$ holds true since
$\omega(\lambda^D_k)\to 0$ as $k\to\infty$ due to Lemma~\ref{rem:maybe-true}.
Finally, \eqref{prop:another-estimate} implies the corresponding assertion on
$\Delta_2(\lambda^D_k)$.
\end{proof}

\subsection{Neumann boundary condition}

The analysis of the asymptotic behavior of $\spec(H^N)$
does not differ much from the Dirichlet case. We start by defining the contours
\begin{equation*}
\cF^m := \left\{z\in\C: \abs{\zeta} = \bigl(m-\tfrac14\bigr)\pi\right\},
\quad
\cF_k := \left\{z\in\C: \abs{\zeta - \bigl(k+\tfrac14\bigr)\pi}=\tfrac{\pi}{2}\right\},
\quad
m,k\in\N.
\end{equation*}
As expected, $\cF_k$ encloses exactly one zero of $\ai'(-\lambda)$
for sufficiently large values of $k$.

\begin{lemma}
\label{lem:bound-for-gA-Neumann}
There exists $m_0, k_0\in\N$ such that, for every $m\ge m_0$ and $k\ge k_0$, the
following statement holds true:
\begin{equation*}
\bigl(1 + \abs{x-z}^{1/4}\bigr) g_A(-z) < 16\sqrt{\pi} \abs{\text{\rm Ai}'(-z)},
\end{equation*}
whenever $z\in\mathcal{F}^m$ or $z\in\mathcal{F}_k$.
\end{lemma}

The proof of this assertion is nearly identical to the proof of
Lemma~\ref{lem:bound-for-gA}, except that it relies on the identities
\begin{align*}
\ai'(z)
	= -z^{1/4} \frac{e^{-\zeta}}{2\sqrt{\pi}}
	\left[1+W_3(z)\right],
	\quad \abs{\arg(z)}\le\tfrac{2\pi}{3},
	\quad \abs{z}\ge 1,
\\[1mm]
\ai'(-z)
	= \frac{z^{1/4}}{\sqrt{\pi}}
	\left[\sin\left(\zeta-\tfrac{\pi}{4}\right)
	+W_4(z)\right],
	\quad \abs{\arg(z)}\le\tfrac{\pi}{3},
	\quad \abs{z}\ge 1,
\end{align*}
where the functions $W_3(z)$ and $W_4(z)$ satisfy
\begin{align*}
\abs{\frac{W_3(z)}{\zeta^{-1}}}\le D_1,
	\quad \abs{\arg(z)}\le\tfrac{2\pi}{3},
	\quad \abs{z}\ge 1,\nonumber
\\[1mm]
\abs{\frac{W_4(z)}
{\zeta^{-1}e^{\abs{\im\zeta}}}}\le D_2,
	\quad \abs{\arg(z)}\le\tfrac{\pi}{3},
	\quad \abs{z}\ge 1.\label{eq:w2-Neumann}
\end{align*}
The details are therefore omitted.

\begin{theorem}
\label{thm-mixed-bc-eigenvalues}
Suppose $q\in L^1(\R_+)$. Then,
the eigenvalues of $H^N$ satisfy
\[
\lambda^N_k = \left(\tfrac32\pi\bigl(k-\tfrac34\bigr)\right)^{2/3}
			\left(1 + O(k^{-1})\right),\quad k\to\infty.
\]
\end{theorem}

\begin{proof}
Since it is similar to the proof of Theorem~\ref{thm-dirichlet-eigenvalues},
we only hint at the main departure from it. Recalling that $\sup_{z\in\C}\omega(z)<\infty$,
\eqref{eq:estimate-diff-phi-neumann} implies
\[
\abs{\psi'(z)-\psi'_0(z)}
	\le C\omega(z)\bigl(1+\abs{z}^{1/4}\bigr)g_A(-z)
\]
for certain positive constant $C$. Because of
Lemma~\ref{lem:bound-for-gA-Neumann},
there exists $k_1\in\N$ such that
\[
\abs{\psi'(z)-\psi'_0(z)} < \abs{\psi'_0(z)}
\]
for all $z\in\cF^{k_1}$ and $z\in\cF_k$ for every $k>k_1$, hence
$\abss{(\lambda^N_k)^{3/2} - (-a'_k)^{3/2}}< \pi$ for all $k$ large enough.
\end{proof}

\begin{theorem}
\label{thm:norming-constant-neumann}
Suppose $q\in L^1(\R_+)$. Then the Neumann norming constants $\nu_k^N$ satisfies
\[
\frac{1}{\nu_k^N}
	= \left(\tfrac32\pi\bigl(k-\tfrac34\bigr)\right)^{2/3}\left(1 + o(1)\right)
\]
as $k\to\infty$.
\end{theorem}

\begin{proof}
The argument goes along the lines of the proof of Theorem~\ref{thm:norming-constant-dirichlet}.
Define
\[
\Delta_3(\lambda)
	:= \frac{\dot{\psi}'(\lambda,0) -
		\dot{\psi}'_0(\lambda,0)}{\sqrt{\pi}\lambda\ai(-\lambda)},
\quad
\Delta_4(\lambda)
	:= \frac{\psi(\lambda,0) - \psi_0(\lambda,0)}
	{\sqrt{\pi}\ai(-\lambda)}.
\]
Then
\[
\frac{\dot{\psi}'(\lambda^N_k,0)}{\psi(\lambda^N_k,0)}
	= \lambda_k^N \frac{1 + \Delta_3(\lambda_k^N)}{1 + \Delta_4(\lambda_k^N)}
\]
so we only need to prove that
\[
\Delta_3(\lambda^D_k)\to 0\quad \text{and}\quad \Delta_4(\lambda^D_k)\to 0
\]
as $k\to\infty$. But this follows from \eqref{eq:estimate-diff-phi} and
\eqref{prop:last-estimate}.
\end{proof}


\begin{acknowledgments}
This research is based upon work supported by Universidad Nacional del Sur (Argentina)
under grant PGI 24/L117.
\end{acknowledgments}



\begin{thebibliography}{X}

\bibitem{ahm2}
	S. Albeverio, R. Hryniv and Ya. Mykytyuk,
	{\em Inverse spectral problems for Bessel operators},
	J. Diff. Eqs. {\bf 241} (2007), 130--159.

\bibitem{calogero}
	F. Calogero and A. Degasperis,
	{\em Inverse spectral problem for the one-dimensional
	Schr\"odinger equation with an additional linear potential},
	Lett. Nuovo Cimento {\bf 23} (1978), 143--149.

\bibitem{car}
	R. Carlson,
	{\em Inverse spectral theory for some singular Sturm--Liouville problems},
	J. Diff. Eqs. {\bf 106} (1993), 121--140.

\bibitem{car2}
	R. Carlson
	 {\em A Borg--Levinson theorem for Bessel operators},
	 Pacific J. Math. {\bf 177} (1997), 1--26.

\bibitem{chelkak1}
	D. Chelkak, P. Kargaev and E. Korotyaev,
	{\em An inverse problem for an harmonic oscillator perturbed by
	potential: uniqueness},
	Lett. Math. Phys. {\bf 64} (2003), 7--21.

\bibitem{chelkak2}
	D. Chelkak, P. Kargaev and E. Korotyaev,
	{\em Inverse problem for harmonic oscillator perturbed by
	potential, characterization},
	Comm. Math. Phys. {\bf 249} (2004), 133--196.

\bibitem{chelkak2.5}
	D. Chelkak and E. Korotyaev,
	{\em The inverse problem for perturbed harmonic oscillator
	on the half-line},
	Institut Mittag–Leffler, Preprint No. 10, 2005/2006 fall.

\bibitem{chelkak3}
	D. Chelkak and E. Korotyaev,
	{\em The inverse problem for perturbed harmonic oscillator on the
	half-line with a {D}irichlet boundary condition},
	Ann. Henri Poincar{\'e} {\bf 8} (2007), 1115--1150.

%

\bibitem{gesztesy}
	F. Gesztesy and B. Simon,
	{\em Uniqueness theorems in inverse spectral theory for one-dimensional
	Schr\"odinger operators},
	Trans. Amer. Math. Soc. {\bf 348} (1996), 349--373.

\bibitem{guliyev}
	N. J. Guliyev,
	{\em Inverse square singularities and eigenparameter dependent
	boundary conditions are two sides of the same coin},
	preprint arXiv:2001.00061 [math-ph].

\bibitem{its}
	A. Its and V. Sukhanov,
	{\em A Riemann--Hilbert approach to the inverse problem for the
	Stark operator on the line},
	Inverse Problems {\bf 32} (2016), 055003 (27pp).

\bibitem{katchalov}
	A. P. Katchalov and Ya. V. Kurylev,
	{\em Inverse scattering problem for a one-dimensional Stark
	effect Hamiltonian},
	Inverse Problems {\bf 6} (1990), L1--L5.

\bibitem{khanmamedov}
A.\ Kh.\ Khanmamedov and M.\ G.\ Makhmudova,
	{\em Inverse spectral problem for the Schrödinger
	equation with an additional linear potential},
	Theor. Math. Phys. {\bf 202} (2020), 58--71.

\bibitem{korot1}
E. L. Korotyaev,
	{\em Asymptotics of resonances for 1D Stark operators},
	Lett. Math. Phys. {\bf 108} (2018), 1307–-1322.

\bibitem{korot2}
E. L. Korotyaev,
	{\em Resonances for 1d Stark operators},
	J. Spectr. Theory {\bf 7} (2017), 699–732.

\bibitem{kst}
A. Kostenko, A. Sakhnovich and G. Teschl,
	{\em Inverse eigenvalue problems for perturbed spherical Schr\"odinger operators},
	Inverse Problems {\bf 26} (2010), 105013 (14pp).

\bibitem{kst2}
A. Kostenko, A. Sakhnovich and G. Teschl,
	{\em Weyl--Titchmarsh theory for Schr\"odinger operators with strongly
	singular potentials},
	Int. Math. Res. Not. {\bf 2012} (2012), 1699--1747.

\bibitem{kt}
A. Kostenko and G. Teschl,
	{\em On the singular Weyl--Titchmarsh function of perturbed spherical
	Schr\"odinger operators},
	J. Differential Equations {\bf 250} (2011), 3701--3739.

\bibitem{kt2}
A. Kostenko and G. Teschl,
	{\em Spectral asymptotics for perturbed spherical Schr\"odinger operators
	and applications to quantum scattering},
	Comm. Math. Phys. {\bf 322} (2013), 255--275.

\bibitem{lk}
A. R. Latifova and A. Kh. Khanmamedov,
	{\em Inverse spectral problem for the one-dimensional
	Stark operator on the semiaxis},
	Ukr. Math. J. {\bf 72} (2020), 568--584.

\bibitem{liu}
W. Liu,
	{\em Criteria for eigenvalues embedded into the absolutely continuous
	spectrum of perturbed Stark type operators},
	J. Funct. Anal. {\bf 276} (2019), 2936--2967.

\bibitem{mk}
M.\ G.\ Makhmudova and A.\ Kh.\ Khanmamedov,
	{\em On spectral properties of the one-dimensional Stark operator
	on the semiaxis},
	Ukr. Math. J. {\bf 71} (2020), 1813--1819.

\bibitem{nist}
	F. W. J. Olver, D. W. Lozier, R. F. Boisvert and
	C. W. Clark,
	{\em NIST Handbook of Mathematical Functions},
	U.S. Department of Commerce National Institute of Standards
	and Technology, Washington D.C., 2010.

\bibitem{poeschel}
	J. Pöschel and E. Trubowitz,
	{\em Inverse Spectral Theory}, Pure and Applied Mathematics vol. 130,
	Academic Press, Boston, 1987.

\bibitem{sukhanov}
	V. V. Sukhanov,
	{\em Trace formulas for the one-dimensional Stark operator and
	integrals of motion for the cylindrical Korteweg–de Vries equation},
	St. Petersburg Math.J. {\bf 31} (2020) 903-910.

\bibitem{weidmann}
	J. Weidmann,
	{\em Spectral Theory of Ordinary Differential Operators}
	Lecture Notes in Mathematics {\bf 1258}, Springer, Berlin, 1987.

\end{thebibliography}
\end{document}